\documentclass{amsart}
\usepackage{amsthm,amssymb,latexsym,amsmath}
\usepackage{mathrsfs}
\usepackage{graphicx}
\usepackage{color}
\usepackage[latin1]{inputenc}
\usepackage[all]{xy}

\newcommand{\CC}{{\mathbb C}}

\renewcommand{\dim}{\mathrm{dim}}

\newcommand{\OO}{\mathcal O}

\newcommand{\J}{\mathscr J}

\newcommand{\Sing}{{\rm Sing}}

\newtheorem{lema}{Lemma}[section]
\newtheorem{cor}[lema]{Corollary}
\newtheorem{teo}[lema]{Theorem}

\newtheorem*{teorema}{Theorem}

\newtheorem{prop}[lema]{Proposition}

\theoremstyle{definition}
\newtheorem{remark}[lema]{Remark}
\newtheorem{defi}[lema]{Definition}

{}
\begin{document}
\title{On Gauss-Bonnet and Poincar\'e-Hopf type  theorems for  complex  $\partial$-manifolds}
\dedicatory{\it To Omegar Calvo-Andrade on the occasion of his 60th birthday}
\begin{abstract}
We prove a  Gauss-Bonnet and Poincar\'e-Hopf type  theorem for  complex  $\partial$-manifold 
$\tilde{X} = X -  D$, where $X$ is a complex compact manifold and $D$ is a reduced divisor. We will consider the cases such that 
$D$ has isolated singularities and also if  $D$  has a  (not necessarily irreducible) decomposition 
$D=D_1\cup D_2$  such that $D_1$, $D_2$  have  isolated   singularities and  
$C=D_1\cap D_2$  is a  codimension $2$ variety  with  isolated   singularities.    
\end{abstract}

\author{Maur\'icio Corr\^ea}
\address{Maur\'icio Corr\^ea \\ Icex - UFMG ,  Av. Ant\^onio Carlos 6627, 30123-970,  Belo Horizonte-MG, Brazil}
\email{mauriciojr@ufmg.br}
\author{Fernando Louren\c co}
\address{ Fernando Louren\c co \\ DEX - UFLA ,  Campus Universit\'ario, Lavras MG, Brazil, CEP 37200-000}
\email{fernando.lourenco@dex.ufla.br}
\author{Diogo Machado}
\address{Diogo Machado \\ DMA - UFV,  Avenida Peter Henry Rolfs, s/n - Campus Universit\'ario, 36570-900 Vi\ cosa- MG,
Brazil}
\email{diogo.machado@ufv.br}
\author{Antonio M. Ferreira}
\address{Antonio M. Ferreira \\ DEX - UFLA ,  Campus Universit\'ario, Lavras MG, Brazil, CEP 37200-000}
\email{antoniosilva@dex.ufla.br}
\maketitle

\section{Introduction}
Let $X$ be a  compact complex manifold of dimension $n$. The classical Chern-Gauss-Bonnet theorem  \cite{chern}  tells us that 
\begin{eqnarray}\label{form001} 
\int_{X}c_n(\Omega_X^1)  = (-1)^n\chi ( X),
\end{eqnarray}
where  $\chi ( X)$ denotes the Euler characteristic  of $X$ and $c_n(\Omega_X^1)$  is the $n$-th Chern class of the cotangent bundle $\Omega_X^1$ of $X$. 
A complex  $\partial$-manifold \cite{YN} is a complex manifold of the form 
$\tilde{X} = X -  D$, where $X$ is an $n$-dimensional complex compact manifold and  $ D\subset X$ is a  divisor
which is called the boundary divisor.
S. Iitaka in \cite{Lita}  proposed a version of Gauss-Bonnet theorem for  $\partial$-manifold \cite{YN}. Such version was independently proved by 
Y. Norimatsu \cite{YN}, R. Silvotti \cite{RS} and P. Aluffi \cite{Aluffi}:

\begin{teorema}[Norimatsu-Silvotti-Aluffi]
Let $\tilde{X}$ be a complex manifold such that $\tilde{X} = X -  D$, where $X$ is an $n$-dimensional complex compact manifold and $ D$ is a hypersurface with normal crossings on $X$. Then 
$$
\int_{X}c_n(\Omega_X^1(\log\,  D)) \,\,\, =\,\,\, (-1)^n\chi (\tilde{X}),
$$
 where $\Omega_X^1(\log\,  D)$ denotes the sheaf of logarithmics $1$-forms along $ D$ and $\chi (\tilde{X})$ denotes the Euler characteristic given by 
$$
\chi (\tilde{X}) = \displaystyle \sum_{i=1}^n \dim \  H^i_{c}(\tilde{X}, \CC).
$$
\end{teorema}

 We recall that a hypersurface $D$ on a complex manifold $X$ of dimension $n$ is  normal crossing if any  irreducible component of  $D$ is  smooth and for every point $p\in X$ a local equation of $D$ is  $z_1   \cdots  z_r$ 
 for independent local parameters   $z_i$ in  $\mathcal{O}_{p,X}$ with  $r\leq n$.

  X. Liao has provided  formulas in \cite{Liao}  in terms of the  Chern-Schwartz-MacPherson class. 
On the other hand, the Poincar\'e-Hopf theorem applied for a compact complex manifold $X$ with a holomorphic  vector field $v\in H^0(X,TX)$, with isolated singularities, gives us  the following 
$$
\chi( X) =\sum_{p\in \Sing(v)\cap  X }PH(v,p),
$$
where $PH(v,p)$ denotes the Poincar\'e-Hopf index of $v$ on $p$.
In \cite{MD,MD3} the first and third named authors  have  proved  the following Poincar\'e-Hopf type theorem for  $\partial$-manifolds with  boundaries  divisors having  normal crossing singularities.   
\begin{teorema}  
Let $\tilde{X}$ be a complex manifold such that $\tilde{X} = X -  D$, where $X$ is an $n$-dimensional  complex compact  manifold, $ D$ is a reduced  normal crossing hypersurface on $X$.  Let $v$ be a holomorphic vector field on $X$, with isolated singularities (non-degenerate) and logarithmic along $ D$. Then 
$$
\chi(\tilde{X}) =\sum_{x\in \Sing(v)\cap  \tilde{X} }PH(v,x),
$$
where $PH(v,x)$ denotes the Poincar\'e-Hopf index of $v$ at  $x$.
\end{teorema}

In this work we provide a Gauss-Bonnet and Poincar\'e-Hopf type theorem for  $\partial$-manifolds 
of the form  $\tilde{X} = X -  D$,  where $X$ is a complex compact manifold and $D$ is a reduced divisor which is not 
normal  crossing.  
More precisely, we   consider the case such that 
$D$ has isolated singularities and also if  $D$  has a  (not necessarily irreducible) decomposition 
$D=D_1\cup D_2$ such that  $D_1$, $D_2$  have  isolated   singularities and  
$C=D_1\cap D_2$  is a  codimension $2$ variety  with  isolated   singularities. 
In this case, 
the sheaf of logarithmic $1$-forms  $\Omega_X^1(\log\,  D)$  is not   locally free, therefore the demonstration of these formulas, by calculating the Chern classes,  requires an appropriate treatment, as we will see in the  subsection \ref{dd002}.

Let us fix some notations before we state our main result: Let $W\subset X$ an analytic subspace and $v\in H^0(X,TX)$  a holomorphic  vector field,
we   denote by
$$
 PH(v,W)=\sum_{x\in  W} PH(v,x), 
$$ 
$$
  \mu( D,W)=\sum_{x\in W}\mu_x( D), 
$$ 
$$
 GSV(v,  D, W)=\sum_{x\in W }  GSV(v,  D, x), 
$$ 
where  $PH(v,x)$ and $GSV(v,  D, x)$ denote, respectively,   the Poincar\'e-Hopf and GSV index of a vector field  $v$ at  $p$ and $\mu_x( D)$ is the Milnor number of $D$ at $x$. 
We refer to Section  2.4 for more details on  GSV index and \cite{JM} for Milnor number.   We also   denote $S(W):=Sing(W)$ and  
  $S(v, W) : = [Sing(v) \cap  W] \cup Sing(W)$. 

We prove  the following  result:

\begin{teo}\label{T1}
Let $\tilde{X}$ be a complex manifold such that $\tilde{X} = X -  D$, where $X$ is an $n$-dimensional ($n\geq 3$)  complex compact manifold and $D$ is a reduced divisor on $X$. Given any (not necessarily irreducible) decomposition $D = D_1\cup D_2$, where $D_1$, $D_2$ have  isolated singularities and $C=D_1\cap D_2$  is a codimension $2$ variety and has isolated singularities,
\begin{itemize}
\item [ (i)] (Gauss-Bonnet type formula) the following  formula  holds\\ 
$$
\int_{X}c_{n}(\Omega^1_X(\log\,  D)) = (-1)^{n}\chi(\tilde{X})+\mu( D_1,S(D_1))+  \mu( D_2,S(D_2))-  \mu( C,S( C)).
$$
\end{itemize}
\bigskip
\begin{itemize}
\item [ (ii)] (Poincar\'e-Hopf type formula) if $v$ is a holomorphic vector field on $X$, with isolated singularities  and logarithmic along $ D$,  we have that 
\\
$$
 \chi(\tilde{X}) = PH(v, \Sing(v)) - GSV(v,  D_1,  S(v, D_1))  - GSV(v,  D_2,  S(v, D_2))  +GSV(v,  C,  S(v, C))  +  
 $$
 $$
+(-1)^{n-1}\left[  \mu( D_1,S(D_1))+  \mu( D_2,S(D_2))-  \mu( C, S( C)) \right].
$$
\end{itemize}
\end{teo}

Considering the particular case where $D$ is a divisor on  $X$  with isolated singularities, we prove the following formulas:

\begin{cor}\label{TT2}
Let $\tilde{X}$ be a complex manifold such that $\tilde{X} = X -  D$, where $X$ is an $n$-dimensional ($n\geq 3$)  complex compact manifold and $ D$ is a reduced divisor on  $X$  with isolated  singularities. Then 
\begin{itemize}
\item [ (i)] (Gauss-Bonnet type formula) the following  formula  holds\\ 
$$
\int_{X}c_{n}(\Omega^1_X(\log\,  D)) = (-1)^{n}\chi(\tilde{X})+\sum_{p\in Sing( D)}\mu_p( D).
$$
\end{itemize}
\bigskip
\begin{itemize}
\item [ (ii)] (Poincar\'e-Hopf type formula) if $v$ is a holomorphic vector field on $X$, with isolated singularities  and logarithmic along $ D$,  we have that 
\\
$$
 \chi(\tilde{X}) =\sum_{x\in \Sing(v)}PH(v,x) - \sum_{x\in S(v, D)}GSV(v,  D, x) +  
(-1)^{n-1}\sum_{x\in Sing( D)}\mu_x( D).
$$
\end{itemize}
Moreover, if the vector field $v$ has  only non-degenerate singularities, then
\\
$$
 \chi(\tilde{X}) =\sum_{x\in \Sing(v) \cap [X -  D_{reg}]}PH(v,x) - \sum_{x\in Sing( D) } \left[ GSV(v,  D, x) +  
(-1)^{n-1} \mu_x( D)\right].
$$
\end{cor}

Finally,  we  proof the following formula 
$$
\chi (\mathbb{P}^n \setminus D) = \displaystyle\sum^{n}_{i=0} (-1)^i(d-1)^i + (-1)^{n+1}  \sum_{p\in Sing( D)}\mu_p( D),
$$
 where $D$ is a reduced divisor with isolated singularities. This result was proved  in (\cite{H}, pg 1537) and (\cite{FM}, Proposition 2.3). 
 Moreover, we prove   the following  generalization.   
\begin{cor}\label{cor-Pn}   Given any (not necessarily irreducible) decomposition $D = D_1\cup D_2$, where $D_1$, $D_2$  have   isolated singularities and $C=D_1\cap D_2$  is a codimension $2$ variety and has isolated singularities, then 
$$
\chi (\mathbb{P}^n \setminus D) =   (-1)^{n} \displaystyle\sum^{n}_{i=0}  \sigma_{n-i}(d_1-1,d_2-1) + (-1)^{n+1}\left[\mu( D_1,S(D_1))+  \mu( D_2,S(D_2))-  \mu( C,S( C))\right],
$$
where  $\sigma_{n-i}$ is the  complete symmetric function of degree $n-i$ and $d_{j}=\deg(D_j)$, for $j=1,2$. 
\end{cor}

\subsection*{Acknowledgments}
We are grateful to A. G. Aleksandrov for interesting and fruitful conversations.
 We would like to thank the
referee for precious comments which improved the presentation of the paper greatly. The first named author was partially supported by   CNPQ grant numbers 202374/2018-1, 302075/2015-1, 400821/2016-8,   and  CAPES  grant number 	
2888/2013; he is  grateful to the University of Oxford for its  hospitality.

\section{Preliminaries}
\subsection{Logarithmic forms and logarithmic vector fields}
Given a complex manifold $X$ of dimension $n$ and $ D$ a reduced hypersurface on $X$.  Let $\Omega^q_{X}( D)$ be the sheaf of differential $q$-forms on $X$ with at most simple poles along $ D$. 

A {\it logarithmic $q$-form along} $ D$ on an open subset $U \subset X$ is a meromorphic $q$-form $\omega$ on $U$, regular on $U -  D$ and such that both $\omega$ and $d\omega$ have at most simple poles along $ D$.
Logarithmic $q$-forms along $ D$ form a coherent sheaf of $\OO_X$-modules denoted by $\Omega_{X}^q(\log\,  D)$.  In this case, for any open subset $U\subset X$ we have 
$$
\Gamma(U,\Omega_{X}^q(\log\,  D)) = \{\omega \in \Gamma(U,\Omega^q_{X}( D)): d\omega \in \Gamma(U,\Omega^{q+1}_{X}( D))\}.
$$
\noindent  See for example \cite{Deligne}, \cite{Katz} and  \cite{Sai}  for more details about the sheaf of logarithmic $q$-forms along $ D$.

Now, consider  $\Omega_{X}^1(\log\,  D)$, the sheaf of logarithmic $1$-forms  along $ D$. Its dual sheaf is the sheaf of logarithmic vector fields along $ D$, denoted by $T_{X}(-\log\,  D)$. We have an exact sequence 

\centerline{
\xymatrix{
0\ar[r]& T_{X}(-\log\,  D) \ar[r]& T_X \ar[r] &  \ar[r]  \J_{D}(D)&0  \\
}
}
\noindent where $\J_{D}$ is the   Jacobian ideal of $D$ which  is defined as the Fitting ideal
$$
\J_{D}:=F^{n-1}(\Omega_D^1) \subset \mathcal{O}_D. 
$$


\hyphenation{theo-ry}

Saito in \cite{Sai} has showed that in general  $\Omega_{X}^1(\log\,  D)$ and $T_{X}(-\log \, D)$  are  reflexive sheaves. 
If  $ D$ is an analytic hypersurface with normal crossing singularities, the sheaves $\Omega_{X}^1(\log\,  D)$ and $T_{X}(-\log \, D)$ are  locally free. Furthermore, the Poincar\'e residue map (see \cite[Section 2]{Sai})

\centerline{
\label{seq1} \xymatrix{
\rm{Res}: \Omega_X^1(\log\,  D) \ar[r]  &  \OO_{ D} \cong \bigoplus_{i=1}^N\OO_{ D_i}  \\
}
}

\noindent  gives us  the following exact sequence of sheaves on $X$: 

\begin{eqnarray}\label{seq1}
\centerline{
\xymatrix{
0\ar[r]& \Omega_X^1 \ar[r]&   \Omega_X^1(\log\,  D)\ar[r]^{\text{Res}} &  \ar[r]   \bigoplus_{i=1}^N\OO_{ D_i}&0,  \\
}
}
\end{eqnarray}
\noindent where $\Omega_X^1$ is the sheaf of holomorphic  $1$-forms on $X$ and $ D_1,\ldots, D_N$ are the irreducible components of $ D$.

Now, if  $ D$ is such that $codim_{X}(\Sing( D))>2$ then there exist the following exact sequence of sheaves on $X$ (see V. I. Dolgachev \cite{IVDnew}):
\begin{eqnarray} \label{seq0002}
\centerline{
\xymatrix{
0\ar[r]& \Omega_X^1 \ar[r]&   \Omega_X^1(\log\,  D)\ar[r]  &    \OO_{ D} \ar[r] &0.  \\
}}
\end{eqnarray}

\hyphenation{di-ffe-ren-tial}

\subsection{Multi-logarithmic forms}

We give here the definition and basic  properties about multi-logarithmic differential forms. For more details and properties we refer \cite{Ale_s} and \cite{POL}.

Suppose $D= D_{1} \cup \cdots \cup D_{k}$
a decomposition of the reduced hypersurface $D$ in the
 complex manifold $X$, where each $D_{i}$ is a hypersurface defined by the holomorphic function $h_{i}$ for $i = 1, \dots ,k$ on an open subset $U \subset X$, and $C = D_{1} \cap \cdots \cap D_{k}$ is a reduced complete intersection.

\begin{remark}
We observe that the decomposition of a divisor $D\subset X$  in our main result (Theorem \ref{T1}) is 
not  necessarily  into irreducible components.  
\end{remark}

A {\it Multi-logarithmic $q$-form along} the complete intersection $C$ on an open subset $U \subset X$ is a meromorphic $q$-form $\omega$ on $U$, regular on $U - D$ and such that $\omega$ have at most simple poles along $C$ and 
$$dh_{i}\wedge\omega  \in \sum_{i=1}^{k}\Omega^{q+1}(\widehat{D_{i}}) \ \ \ \mbox{for all} \ \ i \in \{1, \dots,k \}, $$

\noindent where $\widehat{D_{i}}= D_{1} \cup \cdots \cup D_{i-1}\cup D_{i+1}\cup \cdots \cup D_{k}.$

We denote by $\Omega^{q}_{X}(\log C)$ the coherent sheaf of germs of multi-logarithmic $q$-forms along $C$. A. G. Aleksandrov \cite{Ale_s} has proved the following result that characterizes multi-logarithmic forms.

\begin{teo}[A. G. Aleksandrov, \cite{Ale_s}]

Let $\omega \in \Omega^{q}_{X}(D)$, then $\omega$ is multi-logarithmic along $C$ if, and only if, there is a holomorphic function $g \in \mathcal{O}_{X}$ which is not identically zero on every irreducible component of the $C$, a holomorphic differential form $\xi \in \Omega^{q-k}_{X}$ and a meromorphic $q$-form $\eta \in \sum_{i=i}^{k} \Omega^{q}_{X}(\widehat{D_{i}})$ such that there exists the following representation

$$g\omega = \dfrac{dh_{1}\wedge \cdots \wedge dh_{k}}{h_{1} \cdots h_{k}}\wedge \xi + \eta.$$

\end{teo}

For $q<k$ we have the equality (see \cite{POL}, Remark 2.6): 

$$ \Omega^{q}_{X}(\log C)= \sum_{i=i}^{k} \Omega^{q}_{X}(\widehat{D_{i}}).$$

\noindent Observe that if $q=1$ and $k=2$ we have

\begin{equation}\label{dd01}
\Omega^{1}_{X}(\log C)= \Omega^{1}_{X}(D_{1}) + \Omega^{1}_{X}(D_{2}). 
\end{equation}

\begin{remark}
Theorem \ref{teo 2.4} (and Lemma 2.3)  was demonstrated by A. G. Aleksandrov \cite{Ale_s}  with the hypothesis that $\Omega^{1}(\log D)$ is generated by closed forms, but we observe that this is not necessary for $k=2$.
\end{remark}

\begin{lema}\label{lema 2.3} Let $D= D_{1} \cup D_{2}$ be a reduced hypersurface  on $X$, where $D_{i}$ is a reduced hypersurface, for $i=1,2,$ and $C= D_{1}\cap D_{2}$ is a reduced complete intersection, then
$$\Omega^{1}_{X}(\log D) \subset \Omega^{1}_{X}(\log C).$$ 
\end{lema}

\begin{proof} Let $h_{1}, h_{2}$ and $h = h_{1}h_{2}$ be the equations which  define  $D_{1}, D_{2}$ and $D$, respectively, on an open subset $U \subset X$.
Given $\omega \in \Omega^{1}_{X}(\log D)$  we have that $\theta := dh\wedge \omega$ is holomorphic. Thus,

\begin{equation}\label{eq 02}
\theta = h_{2}dh_{1} \wedge\omega + h_{1}dh_{2}\wedge\omega  \Rightarrow dh_{1}\wedge \omega = \dfrac{\theta}{h_{2}} - h_{1}\dfrac{dh_{2}}{h_{2}}\wedge \omega.
\end{equation}

\noindent Note that $ \theta/h_{2} \in \Omega^{2}_{X}(D_{2})$. Therefore, it is only necessary to prove that $$h_{1}\dfrac{dh_{2}}{h_{2}}\wedge \omega \in \Omega^{2}_{X}(D_{2}).$$  
Since $\Omega_X^{\bullet}(\log D)$ is $\wedge$-closed (see \cite{Sai}, 1.3 (ii)), we have that $\dfrac{dh_{2}}{h_{2}}\wedge \omega \in \Omega^{2}_{X}(\log D)$ and, consequently,  $h_2(h_1\dfrac{dh_{2}}{h_{2}}\wedge \omega) = h\dfrac{dh_{2}}{h_{2}}\wedge \omega$ is holomorphic. Thus, we get $dh_{1}\wedge \omega \in \Omega^{2}_{X}(D_{2})$.  Analogously, we can  show that
$dh_{2}\wedge \omega \in \Omega^{2}_{X}(D_{1})$.

\end{proof}

\begin{prop}\label{teo 2.4} Let $D= D_{1} \cup D_{2}$ be a reduced hypersurface in   $X$, where $D_{i}$ is a reduced hypersurface, for $i=1,2,$ and $C= D_{1}\cap D_{2}$ is a reduced  complete  intersection. Then
$$\Omega^{1}_X(\log D) = \Omega^{1}_X(\log D_{1}) + \Omega^{1}_X(\log D_{2}).$$

\begin{proof} Note that $\Omega^{1}_X(\log D_{1}) + \Omega^{1}_X(\log D_{2}) \subset \Omega^{1}_X(\log D)$. On the other hand, given $\omega \in \Omega^{1}_X(\log D)$, by Lemma \ref{lema 2.3} and equality (\ref{dd01}),  we have  
$$\omega = \dfrac{\theta_{1}}{h_{1}} + \dfrac{\theta_{2}}{h_{2}},$$
 where $\theta_{1}$ and $\theta_{2}$ are holomorphic $1$-forms.  We get
\begin{eqnarray}\nonumber
dh\wedge \omega &=& h_{2}dh_{1}\wedge \dfrac{\theta_{1}}{h_{1}} + h_{2}dh_{1}\wedge \dfrac{\theta_{2}}{h_{2}}+ h_{1}dh_{2}\wedge \dfrac{\theta_{1}}{h_{1}} + h_{1}dh_{2}\wedge \dfrac{\theta_{2}}{h_{2}}=\\\nonumber &=&  dh_{1}\wedge \theta_{2}+ dh_{2}\wedge \theta_{1}+ h_{2}dh_{1}\wedge \dfrac{\theta_{1}}{h_{1}}+h_{1}dh_{2}\wedge \dfrac{\theta_{2}}{h_{2}}. 
\end{eqnarray}

\noindent Consequently, 
$$h_{2}dh_{1}\wedge \dfrac{\theta_{1}}{h_{1}}+h_{1}dh_{2}\wedge \dfrac{\theta_{2}}{h_{2}}$$ is holomorphic. Since $(h_{1},h_{2})$ is a regular sequence, we get $dh_{i}\wedge \theta_{i}=h_{i}\alpha_{i}$ for some holomorphic form $\alpha_{i}$ and $i=1,2$.  Then each $dh_{i} \wedge \dfrac{\theta_{i}}{h_{i}}$ is holomorphic and thus we obtain  
$$\omega = \dfrac{\theta_{1}}{h_{1}} + \dfrac{\theta_{2}}{h_{2}} \in  \Omega^{1}_X(\log D_{1}) + \Omega^{1}_X(\log D_{2}).$$

\end{proof}

\end{prop}

\subsection{ Computations on the Chern class } \label{dd002}

Let $D$ be a reduced divisor on $X$ and consider a decomposition (not necessarily irreducible) $D = D_1 \cup D_2$, where $D_1$ and
$D_2$ have isolated singularities. We have that the sheaves $\Omega^{1}_X(\log D_{1}), \Omega^{1}_X(\log D_{2})$ and $\Omega^{1}_X(\log D)$ are not locally free and therefore not all the properties of the Chern classes for bundles are applicable in here.

In any case, we have (see \cite{Ale_s}, Claim 3) the following sequence 
\begin{eqnarray}\nonumber
\xymatrix{
0\ar[r]& \Omega_X^{1}  \ar[r]&   \Omega^{1}_X(\log D_{1})\oplus \Omega^{1}_X(\log D_{2})  \ar[r]&    \Omega^{1}_X(\log D_{1}) + \Omega^{1}_X(\log D_{2})  \ar[r]&0, \\
}
\end{eqnarray}
and since $ \Omega^{1}_X(\log D_{1}) + \Omega^{1}_X(\log D_{2}) = \Omega^{1}_X(\log D) $, by Proposition  \ref{teo 2.4}, we obtain the exact sequence 
\begin{eqnarray}\label{Log12}
\xymatrix{
0\ar[r]& \Omega_X^{1}  \ar[r]&   \Omega^{1}_X(\log D_{1})\oplus \Omega^{1}_X(\log D_{2})  \ar[r]&   \Omega^{1}_X(\log D)  \ar[r]& 0.} 
\end{eqnarray}

On the other hand, since $\Omega^{1}_X(\log D_{1}), \Omega^{1}_X(\log D_{2})$ and $\Omega^{1}_X(\log D)$ reflexives sheaves (see \cite{Sai}, Corollary  1.7), and using the sequences (\ref{Log12}) we obtain
\begin{eqnarray}\nonumber
c(\Omega^{1}_X)c(\Omega_X^{1}(\log D)) = c(\Omega_X^{1}(\log D_{1})\oplus \Omega_X^{1}(\log D_{2})).
\end{eqnarray}

Thus

\begin{eqnarray}\nonumber
c(\Omega^{1}_X)c(\Omega_X^{1}(\log D)) &=& c(\Omega_X^{1}(\log D_{1})) c(\Omega_X^{1}(\log D_{2})) 
\\\nonumber  &=& c(\Omega_X^{1})c(\OO_{D_1})  c(\Omega_X^{1})c(\OO_{D_2}) , 
\end{eqnarray}
\noindent where in last equality we use the following relations
\begin{eqnarray}\label{c122}
c(\Omega_X^{1}(\log D_{i})) = c(\Omega_X^{1})c(\OO_{D_i}),\,\,\,\, i = 1,2,
\end{eqnarray}
\noindent which can be obtained from the exact sequence (\ref{seq0002}).
Therefore, we obtain the following expression for the  Chern class of the sheaf  $\Omega_X^{1}(\log D)$
\begin{eqnarray}\label{c12}
c(\Omega_X^{1}(\log D)) = c(\OO_{D_{1}}) c(\OO_{D_{2}})c( \Omega_X^{1}),
\end{eqnarray}
\noindent which will be essential in the calculations below.

\subsection{The GSV-Index}
X. G\'omez-Mont, J. Seade and A. Verjovsky \cite{GSV} introduced the  GSV-index for a holomorphic vector field over an analytic hypersurface  with isolated singularities  on a complex manifold,  generalizing the (classical) Poincar\'e-Hopf  index. The GSV-index was extended  for  continuous  vector fields on more general contexts.  J. Seade  and  T. Suwa in \cite{SeaSuw1}  have defined the GSV-index for continuous   vector fields on analytic subvarieties with  isolated complete intersection singularity. J.-P. Brasselet, J. Seade and T. Suwa in \cite{a03}  extended the notion of GSV-index for vector fields defined in certain types of analytical subvariety with non-isolated singularities.

\hyphenation{using}
In \cite{a08}   X. G\'omez-Mont  introduced  the   homological index  of a  holomorphic vector field on an analytic hypersurface with isolated singularities, which coincides with the GSV-index. There is also the {\it virtual index}, introduced by D. Lehmann, M. Soares and T. Suwa \cite{a07}, that via Chern-Weil theory can be interpreted as the GSV-index. M. Brunella \cite{a09} 
 also presents the GSV-index for foliations on complex surfaces  by a different approach and in \cite{MD2} 
  the first and third named authors 
  have  introduced a  GSV type  index  for varieties invariant by
 holomorphic Pfaff systems.

\hyphenation{sin-gu-la-ri-ties}
\hyphenation{sin-gu-la-ri-ty}

Let us recall the  definition of the GSV-index  (\cite{BarSaeSuw}, Ch.3, 3.2). 
Let $ D$ be a hypersurface  with isolated singularities  on an $n$-dimensional complex manifold $X$ and let $v$  be a holomorphic vector field on $X$ with isolated singularities, and logarithmic along $ D$. Given a singular point $x_0 \in Sing( D)$, let $h $  be an analytic function  defining  $ D$ on a neighborhood $U_0$ of $x_0$. The gradient vector field $\overline{grad}\,(h)$ is nowhere vanishing away from $x_0$, because $x_0$ is an isolated singularity.  

Denote by $v_{\ast}$ the restriction of $v$ to the  regular part $ D_{reg} =  D - Sing( D)$ of $D$. On the neighborhood $U_0$,  suppose that  the vector field $v$ is non-singular away from $x_0$. Since $v$ is  logarithmic along $ D$, we have that $\overline{grad}\,(h)(z)$ and $v_{\ast}(z)$ are linearly independent at each point  $z\in U_0 \cap ( D - \{x_0\})$. Assume that $(z_1,\ldots,z_n)$ is a system of complex coordinates on $U_0$ and consider 
 $$S_{\varepsilon} = \{z= (z_1,\ldots,z_n):\,\,\mid\mid z-x_0\mid\mid \,\, = \,\, \varepsilon\}$$  
 the sphere   sufficiently small so that $K =  D \cap S_{\varepsilon} $ is the link of the singularity of $ D$ at $x_0$ (see, for example, \cite{JM}). It is an $(2n-1)$-dimensional real oriented manifold. By using the Gram-Schmidt process, if necessary, the vector fields $v_{\ast}$ and $\overline{grad}\,(h)$  define a continuous map
$$
\phi_{v}:=(v_{\ast},\overline{grad}\,(h)): K \longrightarrow W_{2,n+1}
$$  

\noindent where $W_{2,n+1}$ is the Stiefel manifold of complex $2$-frames in $\CC^{n+1}$. 

\begin{defi}\label{GSV-defi}
The GSV-index of $v$ in $x_0 \in  D$, denoted by $GSV(v, D,x_0)$, is defined as the degree of map $\phi_{v}$. 
\end{defi}

\begin{remark} 
In the definition \ref{GSV-defi} the vector field $v$ can be considered continuous rather than holomorphic.  For more details see  \cite{BarSaeSuw}, \cite{GSV}  and  \cite{Suw2}.
\end{remark}

\begin{remark} \label{rem001}
If $x_0\in D_{reg}$ is a regular point of $ D$, since $v$ logarithmic along $ D$, we have that the Poincar\'e-Hopf index of $v|_{ D}$ in $x_0$ is defined and it coincides with the GSV-index. Then $PH(v,x_0) = PH(v|_{ D},x_0)$ and we have 
\begin{eqnarray}
GSV(v, D,x_0) = PH(v,x_0).
\end{eqnarray}
\end{remark}
\hyphenation{fo-lia-tion}
\section{Proof of the Theorems} In order to prove the Theorem \ref{T1} and  the Corollary \ref{TT2} we will prove the following preliminary result: 

\begin{teo}\label{teo101}
Let $X$ be an $n$-dimensional ($n\geq 3$) complex compact manifold and $D$ a reduced  divisor in $X$.
\begin{itemize}
\item [(i)] If $D = D_1\cup D_2$ is any (not necessarily irreducible) decomposition, where $D_1$, $D_2$ 
have isolated singularities and
 $C=D_1\cap D_2$  is a  codimension $2$ variety and has isolated   singularities, then
\end{itemize}
\begin{eqnarray} \nonumber \int_{X} c_{n}(\Omega^{1}(\log D)) & = & (-1)^{n}\left[ \int_{X} c_{n}(TX) - \int_{D_ {1}}c_{n-1}(TX - [D_{1}]) - \int_{D_ {2}}c_{n-1}(TX - [D_{2}]) \right] + \\\nonumber & & \\ \nonumber   && + \left[ \int_{C} c_{n-2}(TX - [D_{1}]\oplus [D_{2}]) \right]. 
\end{eqnarray}
\begin{itemize}
\item [(ii)] If $D$ 
is a divisor with isolated singularities,  then
\end{itemize}
\begin{eqnarray}\label{for001} 
\int_{X} c_{n}(\Omega_X^{1}(\log  D)  =  (-1)^{n} \left[ \int_{X} c_{n}(T_X) - \int_{ D}c_{n-1}(T_X - [ D]) \right].
\end{eqnarray}
\end{teo}

\begin{proof} We prove the item (i); the proof of (ii) is similar.\\
By the equation (\ref{c12})
\begin{eqnarray}\nonumber
c(\Omega_X^{1}(\log D)) = c(\OO_{D_{1}}) c(\OO_{D_{2}})c( \Omega_X^{1}),
\end{eqnarray}
\noindent we have, 
\begin{eqnarray} \nonumber
\int_{X}c_{n}( \Omega_X^{1}(\log D)) &=& \displaystyle \int_{X}\sum_{i_{1} + i_{2} + i_{3} = n} c_{i_{1}}(\OO_{D_{1}}) c_ {i_{2}}(\OO_{D_{2}})c_{i_{3}}(\Omega_X^{1}) = \\\nonumber && \\\nonumber
&=& \int_{X} c_{n}(\Omega_X^{1}) +  \sum_{\substack{i_{2} + i_{3} = n \\ i_{2}\geq 1}} \int_{X}c_ {i_{2}}(\OO_{D_{2}})c_{i_{3}}(\Omega_X^{1}) + \sum_{\substack{i_{1} + i_{3} = n \\ i_{1}\geq 1}} \int_{X} c_ {i_{1}}(\OO_{D_{1}})c_{i_{3}}(\Omega_X^{1}) + \\\nonumber && \\\nonumber &+& \sum_{\substack{i_{1} + i_{2} + i_{3} = n \\ i_{1},i_{2}\geq 1}} \int_{X} c_{i_{1}}(\OO_{D_{1}}) c_ {i_{2}}(\OO_{D_{2}})c_{i_{3}}(\Omega_X^{1})= \\\label{f}\nonumber && \\\nonumber
&=& (-1)^n\int_{X} c_{n}(T_X) +  \sum_{\substack{i_{2} + i_{3} = n \\ i_{2}\geq 1}} \int_{X}c_1([D_{2}])^ {i_{2}}c_{i_{3}}(\Omega_X^{1}) + \sum_{\substack{i_{1} + i_{3} = n \\ i_{1}\geq 1}} \int_{X} c_1([D_{1}])^ {i_{1}}c_{i_{3}}(\Omega_X^{1}) + \\ \nonumber && \\\nonumber && +\sum_{\substack{i_{1} + i_{2} + i_{3} = n \\ i_{1},i_{2}\geq 1}} \int_{X} c_1([D_{1}])^ {i_{1}} c_1([D_{2}])^ {i_{2}}c_{i_{3}}(\Omega_X^{1}),
\end{eqnarray}
\noindent where in the last step we are using that $c_ {i_{j}}(\OO_{D_{j}})=c_1([D_{j}])^ {i_{j}}$, since $c(\OO_{D_{j}})=c([D_{j}])^{-1}$, for $j=1,2$. 
 
The proof will be finalized by calculating each sum on the right hand side. Indeed, in the first one, by using that $c_1([D_{2}])$ is Poincar\'e  dual to the fundamental class of $D_2$, we obtain
\begin{eqnarray}\nonumber
\sum_{\substack{i_{2} + i_{3} = n \\ i_{2}\geq 1}} \int_{X}c_1([D_{2}])^ {i_{2}}c_{i_{3}}(\Omega_X^{1}) &=& \int_{D_ {2}}\sum_{\substack{i_{2} + i_{3} = n \\ i_{2}\geq 1}}c_{1}([D_{2}])^{i_{2}-1}c_{i_{3}}(\Omega_X^{1}) \\\nonumber && \\\nonumber
&=& \int_{D_ {2}}c_{n-1}(\Omega_X^{1} - [D_{2}]^{\ast})\\\nonumber && \\\nonumber
&=& (-1)^{n-1}\int_{D_ {2}}c_{n-1}(T_X - [D_{2}]),
\end{eqnarray}
\noindent where in the last step we are using the relation between the Chern classes
of a vector bundle and of its dual.
Similarly, we obtain 
\begin{eqnarray}\nonumber
\sum_{\substack{i_{1} + i_{3} = n \\ i_{1}\geq 1}} \int_{X}c_1([D_{1}])^ {i_{2}}c_{i_{3}}(\Omega_X^{1}) = (-1)^{n-1}\int_{D_ {1}}c_{n-1}(T_X - [D_{1}]).
\end{eqnarray}

Finally, the last sum can be calculated by using  that $c_1([D_{1}])c_1([D_{2}])$ is Poincar\'e  dual to the fundamental class of $C = D_1\cap D_2$.  Thus, 
\begin{eqnarray}\nonumber
\sum_{\substack{i_{1} + i_{2} + i_{3} = n \\ i_{1},i_{2}\geq 1}} \int_{X} c_1([D_{1}])^ {i_{1}} c_1([D_{2}])^ {i_{2}}c_{i_{3}}(\Omega_X^{1}) &=& \sum_{\substack{i_{1} + i_{2} + i_{3} = n \\ i_{1},i_{2}\geq 1}} \int_{C} c_1([D_{1}])^ {i_{1}-1} c_1([D_{2}])^ {i_{2}-1}c_{i_{3}}(\Omega_X^{1}) \\\nonumber && \\\nonumber
&=& \int_{C}c_{n-2}(\Omega_X^{1} - [D_{2}]^{\ast} - [D_{1}]^{\ast})\\\nonumber 
&=& (-1)^{n-2} \int_{C}c_{n-2}(T_X - [D_{1}]\oplus[D_{2}]).
\end{eqnarray}
Therefore, we conclude that
\begin{eqnarray} \nonumber \int_{X} c_{n}(\Omega^{1}(\log D)) & = & (-1)^{n}\left[ \int_{X} c_{n}(TX) - \int_{D_ {1}}c_{n-1}(TX - [D_{1}]) - \int_{D_ {2}}c_{n-1}(TX - [D_{2}]) \right] + \\\nonumber & & \\\nonumber   && + \left[ \int_{C}c_{n-2}(TX - [D_{1}]\oplus [D_{2}]) \right]. 
\end{eqnarray}

\end{proof}

\noindent {\bf Proof of Theorem \ref{T1}:} Using the classical Chern-Gauss-Bonnet formula (\ref{form001}), we obtain
\begin{eqnarray}\label{x06}
\int_{X}c_n(T_X) = (-1)^n\int_{X}c_n(\Omega_X^1) = \chi ( X).
\end{eqnarray}
From \cite[Theorem 3.9]{Suw2}, we have that
\begin{eqnarray}\label{x01}
\displaystyle\int_{D_i} c_{n-1}(T_X - [D_i]) \, = \, \chi(D_i)\,\, +\,\,(-1)^{n-1}\sum_{p\in Sing(D_i)}\mu_p(D_i), \ \ i=1,2.
\end{eqnarray}

Moreover, since $C = D_{1} \cap D_{2}$ is a   complete intersection $C = D_{1} \cap D_{2}$, its  normal bundle is  $([D_{1}]\oplus [D_{2}])|_{C}$,   and once again from  \cite[Theorem 3.9]{Suw2}) we have that 
\begin{eqnarray}\label{x02}
\int_{C}c_{n-2}(T_X - [D_{1}]\oplus [D_{2}]) \,= \, \chi(C)\,\, +\,\, (-1)^{n-2}\sum_{p\in Sing(C)}\mu_p(C). 
\end{eqnarray}

Now, substituting (\ref{x06}),  (\ref{x01}) and (\ref{x02}) in the formula of item $(i)$ of  Theorem \ref{teo101}, we get the desired formula
\begin{eqnarray}\label{x04}
\int_{X}c_{n}(\Omega_X^1(\log\, D)) &=& \\\nonumber \\\nonumber (-1)^{n}\chi(\tilde{X})   + \sum_{p\in Sing(D_1)}\mu_p(D_1) + \sum_{p\in Sing(D_2)}\mu_p(D_2) + \sum_{p\in Sing(C)}\mu_p(C).
\end{eqnarray}

On the other hand, if $v$ is a holomorphic vector field on $X$, with isolated singularities and logarithmic along $D_1, D_2$ and $C$, it follows from  \cite[Theorem 7.16]{Suw2} that for each $i=1,2$ 
\begin{eqnarray}\label{z02}
\displaystyle\int_{D_i} c_{n-1}(T_X - [D_i]) = \sum_{x\in S(v, D_i) }GSV(v, D_i, x)
\end{eqnarray}
and
\begin{eqnarray}\label{z03}
\displaystyle\int_{C} c_{n-2}(T_X - [D_{1}]\oplus[D_{2}]) = \sum_{x\in S(v, C) }GSV(v, C, x),
\end{eqnarray}
\noindent where $S(v, C) = [Sing(v) \cap  C ] \cup Sing(C)$ and $GSV(v, C, x)$ denotes the GSV-index of $v$ (relative to $C$) at $x$. Thus, using the classical Poincar\'e-Hopf Theorem, (\ref{z02}) and (\ref{z03}) in the item (i) of Theorem \ref{teo101} we get
\begin{eqnarray}\nonumber
(-1)^n\int_{X}c_{n}(\Omega_X^1(\log\, D))=
\end{eqnarray}
\begin{eqnarray}\nonumber
= \sum_{x\in \Sing(v)}PH(v,x) - \sum_{x\in S(v, D_1) }GSV(v, D_1, x) - \sum_{x\in S(v, D_2) }GSV(v, D_2, x) + \sum_{x\in S(v, C) }GSV(v, C, x) . 
\end{eqnarray}

Now, replacing it in the formula (\ref{x04}), we get
\begin{eqnarray}\nonumber
\chi(\tilde{X}) =  \sum_{x\in \Sing(v)}PH(v,x) - \sum_{x\in S(v, D_1) }GSV(v, D_1, x) - \sum_{x\in S(v, D_2) }GSV(v, D_2, x) + \\\nonumber \\\nonumber +  \sum_{x\in S(v, C) }GSV(v, C, x)    + (-1)^{n-1}\left[ \sum_{p\in Sing(D_1)}\mu_p(D_1) + \sum_{p\in Sing(D_2)}\mu_p(D_2) + \sum_{p\in Sing(C)}\mu_p(C)\right].
\end{eqnarray}
$\square$

The proof of Corollary  \ref{TT2}  uses the same techniques.

\section{Proof of Corollary \ref{cor-Pn}}

First of all we proof the following formula

$$
\chi (\mathbb{P}^n \setminus D) = \displaystyle\sum^{n}_{i=0} (-1)^i(d-1)^i + (-1)^{n+1}  \sum_{p\in Sing( D)}\mu_p( D).
$$

Indeed, it  follows for Corollary \ref{TT2} that
$$
 (-1)^{n} \chi (\mathbb{P}^n \setminus D) =  \int_{\mathbb{P}^n}c_{n}(\Omega^1_{\mathbb{P}^n}(\log\,  D)) -  \sum_{p\in Sing( D)}\mu_p( D).  
$$
The  total Chern classe of  $ c(\Omega^1_{\mathbb{P}^n}(\log\,  D))$ is 
 $$
 c(\Omega^1_{\mathbb{P}^n}(\log\,  D))=\frac{c(\Omega^1_{\mathbb{P}^n})}{c(\mathcal{O}_{\mathbb{P}^n}(-d))}=c(\Omega^1_{\mathbb{P}^n}-\mathcal{O}_{\mathbb{P}^n}(-d)).
 $$
 In particular,  $ c_n(\Omega^1_{\mathbb{P}^n}(\log\,  D))= c_n(\Omega^1_{\mathbb{P}^n}-\mathcal{O}_{\mathbb{P}^n}(-d))= c_n(\Omega^1_{\mathbb{P}^n}\otimes \mathcal{O}_{\mathbb{P}^n}(d)))$.
 Since $c_n(\Omega^1_{\mathbb{P}^n}\otimes \mathcal{O}_{\mathbb{P}^n}(d)))=(-1)^{n}c_n(T_{\mathbb{P}^n}\otimes \mathcal{O}_{\mathbb{P}^n}(-d)))$, we conclude that 
   $$ c_n(\Omega^1_{\mathbb{P}^n}(\log\,  D))=(-1)^{n}c_n(T_{\mathbb{P}^n}\otimes \mathcal{O}_{\mathbb{P}^n}(-d))).$$
We have 
 $$
 c_n(T_{\mathbb{P}^n}\otimes \mathcal{O}_{\mathbb{P}^n}(-d)))=    \displaystyle\sum^{n}_{i=0} (1-d)^i h^n= \displaystyle\sum^{n}_{i=0} (-1)^i(d-1)^ih^n.
 $$
Thus,
 we get 
$$
 (-1)^{n} \chi (\mathbb{P}^n \setminus D) =  (-1)^{n} \displaystyle\sum^{n}_{i=0} (-1)^i(d-1)^i-  \sum_{p\in Sing( D)}\mu_p( D) .
$$

The general formula that appears  in Corollary \ref{cor-Pn} can be computed similarly using the Theorem (\ref{T1}) and the equation (\ref{c12}) which gives us  
$$
 c_n(\Omega^1_{\mathbb{P}^n}(\log\,  D_{1}\cup D_{2}))=\left[\frac{c(\Omega^1_{\mathbb{P}^n})}{c(\mathcal{O}_{\mathbb{P}^n}(-d_1)) c(\mathcal{O}_{\mathbb{P}^n}(-d_2))}\right]_n=
 \left[\frac{(1-h)^{n+1}}{(1-d_1h)(1-d_2h)}\right]_n.
 $$
It follows from  \cite[Proposition 4.4]{CSV} that 
$$
 \left[\frac{(1-h)^{n+1}}{ (1-d_1h)(1-d_2h)}\right]_n= \displaystyle\sum^{n}_{i=0}   \sigma_{n-i}(d_1-1,   d_2-1)h^n,
 $$
 where  $\sigma_{n-i}$ is the  complete symmetric function of degree $n-i$.

\end{document}